\documentclass[12pt,reqno]{amsart}
\usepackage{amssymb,a4wide,eucal,amscd,yhmath,amsgen,amsopn}
\usepackage{fancyhdr}
\usepackage{amsmath,hyperref,xcolor}
\usepackage{mathrsfs}
\usepackage[all,cmtip]{xy}
\SilentMatrices
\SelectTips{eu}{}
\newtheorem{theorem}{Theorem}[section]
\newtheorem{proposition}[theorem]{Proposition}

\newtheorem{lemma}[theorem]{Lemma}
\newtheorem{Not}[theorem]{Notation}
\theoremstyle{rem}

\theoremstyle{definition}
\newtheorem{definition}[theorem]{Definition}
\theoremstyle{construct}
\newtheorem{construct}[theorem]{Construction}
\theoremstyle{examp}

\newcommand\projective\mathbf
\newcommand\PP{\projective P}

\newcommand\OO{\mathcal O}

\newcommand\ZZ{\mathbb Z}

\newcommand\onto\twoheadrightarrow
\newcommand\lra\longrightarrow
\newcommand\dar\downarrow
\DeclareMathOperator{\pic}{Pic}
\DeclareMathOperator{\im}{im}
\DeclareMathOperator{\cok}{coker}
\DeclareMathOperator{\rk}{rank}
\DeclareMathOperator{\Hom}{Hom}

\begin{document}

\title{Monads on Cartesian products of projective spaces}
\author{Damian M Maingi}
\date{June, 2023}

\address{Department of Mathematics\\ School of Physical Sciences\\ Chiromo Campus\\ Chiromo Way\\University of Nairobi\\P.O Box 30197, 00100 Nairobi, Kenya\\
Department of Mathematics\\Catholic University of Eastern Africa\\P.O Box 62157, 00200 Nairobi, Kenya}
\email{dmaingi@uonbi.ac.ke, dmaingi@cuea.edu, dmaingi@squ.edu.om}

\maketitle

\begin{abstract}
In this paper we establish the existence of monads on special Cartesian products of projective spaces. Special in the sense that they inject 
onto an odd dimension projective space.  We first construct monads on $\PP^1\times\cdots\times\PP^1\times\PP^3\times\cdots\times\PP^3\times\PP^5\times\cdots\times\PP^5$.
We then proceed to prove stability of the kernel bundle associated to the monad and simplicity of the cohomology vector bundle.
Lastly we establish the existence of monads on $\PP^{a_1}\times\cdots\times\PP^{a_n }$ where $a_1<a_2<\ldots<a_n$, alternating
even and odd or at least $a_i$ $0<i\leq{n}$ is odd.
\end{abstract}

\section{Introduction}

\noindent 
A monad of sheaves on a variety $X$ is a sequence of sheaves $\xymatrix{0\ar[r] & A \ar[r] & B \ar[r] & C \ar[r] & 0}$ on $X$
that happens to be exact at $A$ and at $C$. They are very important methods or tools to construct indecomposable vector bundles with prescribed invariants like rank, chern class etc. Monads appear in many contexts within algebraic geometry; 
first, Horrocks \cite{7} showed that all vector bundles $E$ on $\PP^2$ and $\PP^3$ admit double ended resolutions by line bundles which he called monads.
He proved that all vector bundles $E$ on $\PP^3$ could be obtained as the cohomology bundle of a monad of the form
\[
\xymatrix
{
0\ar[r] & \oplus_i\OO_{\PP^3}(a_i) \ar[r]^{f} & \oplus_j\OO_{\PP^3}(b_j) \ar[r]^{g} & \oplus_n\OO_{\PP^3}(c_n) \ar[r] & 0
}
\]
where $f$ and $g$ are matrices whose entries are homogeneous polynomials of degrees $b_j-a_i$ and $c_n-b_j$ respectively for some integers $i,j,n$.
Barth and Hulek\cite{2} reproved Horrocks showing the uniqueness of the monads obtained which are useful in construction of moduli of stable vector bundles.
A great source of motivation proceeds from Hartshorne's list\cite{5} of problems some of which inquire about the existence of low rank  indecomposable vector bundles.

\vspace{1cm}

\noindent The goal of this paper is the construction of monads on Cartesian products of projective spaces. 
The construction of monads on projective spaces $\PP^k$ was proved by Fl\o{}ystad\cite{4}. Marchesi et al \cite{15} generalized this further for more generalized projective varieties. 
Costa and Miro\cite{3} established existence of monads on smooth hyperquadrics. Maingi established the existence of monads on $\PP^{n}\times\PP^{m}$  \cite{10}, on $\PP^{2n+1}\times\PP^{2n+1}$  \cite{11},
he established monads on  $\PP^{a_1}\times\PP^{a_1}\times\PP^{a_2}\times\PP^{a_2}\times\cdots\times\PP^{a_n}\times\PP^{a_n}$ \cite{12} a generalization
of his work \cite{13} where he established existence of monads on $\PP^{n}\times\PP^{n}\times\PP^{m}\times\PP^{m}$ and recently \cite{14}
for a polarisation $\OO_X(\alpha_1,\cdots,\alpha_t)$ he established existence of monads on $X=\PP^{a_1}\times\cdots\times\PP^{a_n}$.
\vspace{1cm}

\noindent In this paper we establish the existence of monads 
\[
\begin{CD}
M_\bullet: 0@>>>\OO_{X}(-1,\cdots,-1)^{\oplus{k}}@>>^{f}>{\OO^{\oplus{2\mu}\oplus{2k}}_X} @>>^{g}>\OO_{X}(1,\cdots,1)^{\oplus{k}}  @>>>0\\
\end{CD}
\]
on a Cartesian product of spaces $X = \PP^1\times\cdots\times\PP^1\times\PP^3\times\cdots\times\PP^3\times\PP^5\times\cdots\times\PP^5$,
where $l,m,n,k$ are positive integers and $\mu=2^{l+2m+n-1}3^n-1$. We prove stability of the kernel bundle and simplicity of the cohomology bundle by analysis of the display diagram 
of the monad, twisting by line bundles and taking cohomology appropriately.\\
\\
\noindent Next we also look into the case of monads for Cartesian products $X = \PP^{a_1}\times\cdots\times\PP^{a_n}$ where $a_1\leq a_2\leq\cdots\leq a_n$ alternating even and odd or at least $a_i$ is odd for $0<i\leq n$.
We construct the monad \[
\begin{CD}
0@>>>\OO_{X}(-1,-1,-1)^{\alpha} @>>^{\overline{A}}>{\OO^{\beta}_{X}} @>>^{\overline{B}}>\OO_{X}(1,1,1)^{\gamma}  @>>>0\\
\end{CD}
\] on $\PP^1\times\PP^2\times\PP^3$.\\

\noindent As we mimick constructions on instanton bundle and generalize the results, the flow of the paper is somewhat similar to a paper by Ancona and Ottaviani \cite{1} where they proved that special instanton bundles on $\PP^{2n+1}$, 
papers by Maingi\cite{10,11,12,13,14}, here the methods used generalize methods previously used  by several authors for odd dimensional projective spaces.\\

\vspace{1cm}

\begin{Not}
\noindent Since the ambient space in this work is $\PP^1\times\cdots\times\PP^1\times\PP^3\times\cdots\times\PP^3\times\PP^5\times\cdots\times\PP^5$ then $\pic(X) \simeq \ZZ^{l+m+n}$.\\
We shall denote the generators of the Picard group of $X$, $\pic(X)$ by $f_i,g_j$ and $h_k$ where $i=1\cdots,l$, $j=1\cdots,m$ and $k=1\cdots,n$  \\
\\
Next, we denote by $\OO_X(f_1,\cdots,f_l,g_1,\cdots,g_m,h_1,\cdots,h_n):= {p_1}^*\OO_{\PP^1}(f_1)\otimes\cdots\otimes {p_l}^*\OO_{\PP^1}(f_l)\otimes
{q_1}^*\OO_{\PP^3}(g_1)\otimes\cdots\otimes {q_m}^*\OO_{\PP^3}(g_m)\otimes{s_1}^*\OO_{\PP^5}(h_1)\otimes\cdots\otimes {s_n}^*\OO_{\PP^5}(h_n)$,
where $p_i$ for $i=1,\cdots,l$ are natural projections  from $X$ onto $\PP^1$, 
$q_j$ for $j=1,\cdots,m$ are natural projections  from $X$ onto $\PP^3$,
$s_k$ for $k=1,\cdots,n$ are natural projections  from $X$ onto $\PP^5$.\\
\\
For the line bundle $\mathscr{L} = \OO_X(f_1,\cdots,f_l,g_1,\cdots,g_m,h_1,\cdots,h_m)$ on $X$ and a vector bundle $E$, we write 
$E(f_1,\cdots,f_l,g_1,\cdots,g_m,h_1,\cdots,h_n) = E\otimes\OO_X(f_1,\cdots,f_l,g_1,\cdots,g_m,h_1,\cdots,h_n)$  and
$\displaystyle{(f_1,\cdots,f_l,g_1,\cdots,g_m,h_1,\cdots,h_n):= \sum_{i=1}^l1\cdot[f_i\times\PP^1]+\sum_{j=1}^m1\cdot[g_j\times\PP^3]+\sum_{k=1}^n1\cdot[h_k\times\PP^5]}$ representing its corresponding divisor.\\
\\
\noindent We define the normalization of a vector bundle $E$ on $X$ with respect to a polarization $\mathscr{L}$ as follows:\\
\\
We have $\deg_{\mathscr{L}}(E(-k_E,0,\cdots,0))=\deg_{\mathscr{L}}(E)-(l+m+n)k\cdot \rk(E)$ and so we set $d=\deg_{\mathscr{L}}(\OO_X(1,0,\cdots,0))$,
there exists a unique integer $k_E$ given by  $k_E:=\lceil\mu_\mathscr{L}(E)/d\rceil$ such that  $1 - d.\rk(E)\leq \deg_\mathscr{L}(E(-k_E,0,\cdots,0))\leq0$. 
The $\mathscr{L}$-normalization of the bundle $E$ is  the twisted bundle $E_{{\mathscr{L}}-norm}:= E(-k_E,0,\cdots,0)$.\\
\\
The linear functional $\delta_{\mathscr{L}}$ on  $\mathbb{Z}^{l+m+n}$ is defined as 
$\delta_{\mathscr{L}}(f_1,\cdots,f_l,g_1,\cdots,g_m,h_1,\cdots,h_n):= \deg_{\mathscr{L}}\OO_{X}(f_1,\cdots,f_l,g_1,\cdots,g_m,h_1,\cdots,h_n)$.\\
\\
In many cases for the sake of brevity we shall use the notation $H^p(\mathscr{E})$ in place of $H^p(X,\mathscr{E})$ for the $p^{-th}$ cohomology group. 
\end{Not}

\section{Preliminaries}

\noindent In this section we give a set up for the tools that we use in the entire paper. 
Most of the definitions can be found in Chapter 2 of the book by Okonek C et al\cite{16}.

\begin{definition}
Let $X$ be a nonsingular projective variety. 
\begin{enumerate}
\renewcommand{\theenumi}{\alph{enumi}}
 \item A {\it{monad}} on $X$ is a complex of vector bundles:
\[
\xymatrix{0\ar[r] & M_0 \ar[r]^{\alpha} & M_1 \ar[r]^{\beta} & M_2 \ar[r] & 0}
\]
which is exact at $M_0$ and at $M_2$ i.e. $\alpha$ is injective and $\beta$ surjective.
\item A monad as defined above has a display diagram of short exact sequences as shown below:
\[
\begin{CD}
@.@.0@.0\\
@.@.@VVV@VVV\\
0@>>>{M_0} @>>>\ker{\beta}@>>>E@>>>0\\
@.||@.@VVV@VVV\\
0@>>>{M_0} @>>^{\alpha}>{M_1}@>>>\cok{\alpha}@>>>0\\
@.@.@V^{\beta}VV@VVV\\
@.@.{M_2}@={M_2}\\
@.@.@VVV@VVV\\
@.@.0@.0
\end{CD}
\]
\item The vector bundle $E = \ker(\beta)/\im (\alpha)$ and is called the cohomology bundle of the monad.
\end{enumerate}
\end{definition}

\begin{definition}
Let $X$ be a non-singular irreducible projective variety of dimension $d$ and let $\mathscr{L}$ be an ample line bundle on $X$. For a 
torsion-free sheaf $F$ on $X$ we define
\begin{enumerate}
\renewcommand{\theenumi}{\alph{enumi}}
 \item the degree of $F$ relative to $\mathscr{L}$ as $\deg_{\mathscr{L}}F:= c_1(F)\cdot \mathscr{L}^{d-1}$, where $c_1(F)$ is the first Chern class of $F$
 \item the slope of $F$ as $\mu_{\mathscr{L}}(F):= \frac{\deg_{\mathscr{L}}F}{rk(F)}$.
  \end{enumerate}
\end{definition}

\subsection{Hoppe's Criterion over polycyclic varieties.}
Suppose that the Picard group Pic$(X) \simeq \ZZ^l$ where $l\geq2$ is an integer then $X$ is a polycyclic variety.
Given a divisor $B$ on $X$ we define $\delta_{\mathscr{L}}(B):= \deg_{\mathscr{L}}\OO_{X}(B)$.
Then one has the following stability criterion which is a generalization of Hoppe's criterion  for a cyclic variety\cite{6}.

\begin{theorem}[Generalized Hoppe Criterion, \cite{9}, Theorem 3]
 Let $G\rightarrow X$ be a holomorphic vector bundle of rank $r\geq2$ over a polycyclic variety $X$ equiped with a polarisation 
 $\mathscr{L}$ if
 \[H^0(X,(\wedge^sG)\otimes\OO_X(B))=0\] 
 for all $B\in\pic(X)$ and $s\in\{1,\ldots,r-1\}$ such that
 $\begin{CD}\displaystyle{\delta_{\mathscr{L}}(B)<-s\mu_{\mathscr{L}}(G)}\end{CD}$ then $G$ is stable and if
 $\begin{CD}\displaystyle{\delta_{\mathscr{L}}(B)\leq-s\mu_{\mathscr{L}}(G)}\end{CD}$ then $G$ is semi-stable.\\
\\
 Conversely if then $G$ is (semi-)stable then  \[H^0(X,G\otimes\OO_X(B))=0\]
 for all $B\in\pic(X)$ and all $s\in\{1,\ldots,r-1\}$ such that
 $\left(\delta_{\mathscr{L}}(B)\leq\right)$ $\delta_{\mathscr{L}}(B)<-s\mu_{\mathscr{L}}(G)$.
\end{theorem}

\vspace{1cm}

\noindent Next we customize existing results for the purposes of our work herein specifically to Cartesian products projectives.

\begin{proposition}
Given a polycyclic variety $X$ that has Picard number $l+m+n$, an ample line bundle $\mathscr{L}$ and a holomorphic vector bundle $T$ of rank $r>1 $ over $X$,\\
if $H^0(X,(\bigwedge^q T)_{{\mathscr{L}}-norm}(f_1,\cdots,f_{l},g_1,\cdots,g_m,h_1,\cdots,h_n)) = 0$ for $1\leq q \leq r-1$ and every $(f_1,\cdots,f_l,g_1,\cdots,g_m,h_1,\cdots,h_n)\in \mathbb{Z}^{l+m+n}$ such that $\delta_{\mathscr{L}}\leq0$
then $T$ is $\mathscr{L}$-stable.
\end{proposition}

\begin{proposition}
Let $0\rightarrow E \rightarrow F \rightarrow G\rightarrow0$ be an exact sequence of vector bundles.
Then we have the following exact sequence involving exterior and symmetric powers
\[0\lra\bigwedge^q E \lra\bigwedge^q F \lra\bigwedge^{q-1} F\otimes G\lra\cdots \lra F\otimes S^{q-1}G \lra S^{q}G\lra0\]
\end{proposition}

\begin{theorem}[K\"{u}nneth formula]
 Let $X$ and $Y$ be projective varieties over a field $k$. 
 Let $\mathscr{F}$ and $\mathscr{G}$ be coherent sheaves on $X$ and $Y$ respectively.
 Let $\mathscr{F}\boxtimes\mathscr{G}$ denote $p_1^*(\mathscr{F})\otimes p_2^*(\mathscr{G})$\\
 then $\displaystyle{H^m(X\times Y,\mathscr{F}\boxtimes\mathscr{G}) \cong \bigoplus_{p+q=m} H^p(X,\mathscr{F})\otimes H^q(Y,\mathscr{G})}$.
\end{theorem}

\begin{lemma}
Let $\PP^1\times\cdots\times\PP^1\times\PP^3\times\cdots\times\PP^3\times\PP^5\times\cdots\times\PP^5$ then\\
$\displaystyle{H^t(X,\OO_X (f_1,\cdots,f_l,g_1,\cdots,g_m,h_1,\cdots,h_n))\cong \bigoplus_{\sum_{i=1}^l{x_i}+\sum_{j=1}^m{y_j}+\sum_{k=1}^l{z_k}=t} X\otimes Y\otimes Z}$
\\
Where $X=H^{x_1}(\PP^1,\OO_{\PP^1}(f_1))\otimes \cdots \otimes H^{x_l}(\PP^1,\OO_{\PP^1}(f_l))$ \\
$Y=H^{y_1}(\PP^3,\OO_{\PP^3}(g_1))\otimes \cdots \otimes H^{y_m}(\PP^3,\OO_{\PP^3}(g_m))$ \\
$Z=H^{z_1}(\PP^5,\OO_{\PP^5}(h_1))\otimes \cdots \otimes H^{z_n}(\PP^5,\OO_{\PP^5}(h_n))$ 
 \end{lemma}

\begin{theorem}[\cite{17}, Theorem 4.1]
 Let $n\geq1$ be an integer  and $d$ be an integer. We denote by $S_d$ the space of homogeneous polynomials of degree $d$ in 
 $n+1$ variables (conventionally if $d<0$ then $S_d=0$). Then the following statements are true:
 \begin{enumerate}
 \renewcommand{\theenumi}{\alph{enumi}}
  \item $H^0(\PP^n,\OO_{\PP^n}(d))=S_d$ for all $d$.
  \item $H^i(\PP^n,\OO_{\PP^n}(d))=0$ for $1<i<n$ and for all $d$.
  \item $H^n(\PP^n,\OO_{\PP^n}(d))\cong H^0(\PP^n,\OO_{\PP^n}(-d-n-1))$.
 \end{enumerate}
\end{theorem}

\begin{lemma}
If  $\displaystyle{\sum_{i=1}^lf_i>}0$, $\displaystyle{\sum_{i=1}^mg_i>}0$, $\displaystyle{\sum_{i=1}^n h_i>}0$ then\\
$h^p(X,\OO_X (-f_1,\cdots,-f_l,-g_1,\cdots,-g_m,-h_1,\cdots,-h_n)^{\oplus k}) = 0$ where \\
$\PP^1\times\cdots\times\PP^1\times\PP^3\times\cdots\times\PP^3\times\PP^5\times\cdots\times\PP^5$ and for $0\leq p< \dim(X)$, for $k$ a positive integer.
\end{lemma}

\begin{lemma}[\cite{8}, Lemma 10]
Let $A$ and $B$ be vector bundles canonically pulled back from $A'$ on $\PP^n$ and $B'$ on $\PP^m$ then\\
$\displaystyle{H^q(\bigwedge^s(A\otimes B))=
\sum_{k_1+\cdots+k_s=q}\big\{\bigoplus_{i=1}^{s}(\sum_{j=0}^s\sum_{m=0}^{k_i}H^m(\wedge^j(A))\otimes(H^{k_i-m}(\wedge^{s-j}(B)))) \big\}}$.
\end{lemma}

\begin{lemma}[\cite{4}, Main Theorem] Let $\nu\geq1$. There exists monads on $\PP^{\nu}$ whose maps are matrices of linear forms,
\[
\begin{CD}
0@>>>{\OO_{\PP^{\nu}}(-1)^{\oplus a}} @>>^{A}>{\OO^{\oplus b}_{\PP^{\nu}}} @>>^{B}>{\OO_{\PP^{\nu}}(1)^{\oplus c}} @>>>0\\
\end{CD}
\]
if and only if at least one of the following is fulfilled;\\
$(1)b\geq2c+\nu-1$ , $b\geq a+c$ and \\
$(2)b\geq a+c+\nu$
\end{lemma}

\begin{lemma}[\cite{11}, Theorem 3.9]
Let $n$ and $k$ be positive integers and $A$ and $B$ be morphisms of linear forms as in 
\[ B :=\left( \begin{array}{cccc|cccccccc}
x_0\cdots  & x_n &       &   &y_0 \cdots  & y_n\\
    &\ddots&\ddots &&\ddots&\ddots\\
    && x_0\cdots   x_n & & & y_0 \cdots  & y_n
\end{array} \right)
\]
 and
\[ A :=\left( \begin{array}{cccccccc}
-y_0\cdots  & -y_n \\
    	   &\ddots &\ddots\\
             &&-y_0 \cdots & -y_n\\
\hline
x_0 \cdots  & x_n \\
    	    &\ddots &\ddots\\
             && x_0\cdots & x_n\\
\end{array} \right)
\]
 then there exists a  linear monad of the form
\[
\begin{CD}
0@>>>\OO_{\PP^{2n+1}}(-1)^{\oplus k} @>>^{{A}}>{\OO^{\oplus2n+2k}_{\PP^{2n+1}}} @>>^{{B}}>\OO_{\PP^{2n+1}}(1)^{\oplus k} @>>>0\\
\end{CD}
\]
\end{lemma}

\vspace{1cm}

\section{Monads on Cartesian products of $\PP^1$,  $\PP^3$ and  $\PP^5$}

\noindent In the this section we construct explicitly the morphisms that establish the existence of the monad in Theorem 3.2 which forms part of the main result of this paper. 
Thereafter, we study the associated vector bundles proving stability of the kernel bundle and simplicity of the cohomology vector bundle.
\begin{construct}

Let $\psi : X = \PP^1\times\cdots\times\PP^1\times\PP^3\times\cdots\times\PP^3\times\PP^5\times\cdots\times\PP^5\hookrightarrow \PP^{N=2\mu+1}$ be the Segre embedding which is defined as follows:\\
\[\psi([\alpha_{10}:\alpha_{11}]:\ldots:[\alpha_{l0}:\alpha_{l1}][\beta_{10}:\beta_{11}:\beta_{12}:\beta_{13}]:\ldots:(\beta_{m0}:\beta_{m1}:\beta_{m2}:\beta_{m3})\]
\[(\gamma_{10}:\gamma_{11}:\gamma_{12}:\gamma_{13}:\gamma_{14}:\gamma_{15}):\ldots:[\gamma_{n0}:\gamma_{n1}:\gamma_{n2}:\gamma_{n3}:\gamma_{n4}:\gamma_{n5}])\in X\]
\[
\begin{CD}
@.@V^{injects}VV\\
@.[x_0:x_1:\cdots:x_\mu:y_0:y_1:\ldots:y_\mu]\in\PP^{2\mu+1}\\
\end{CD}
\]

\noindent First note that since we are taking $l$ copies of $\PP^1$,  $m$ copies of $\PP^3$ and $n$ copies of $\PP^5$ then we have \\
$N=2^l4^m6^n-1$\\
$=2^l2^{2m}2^n3^n-1$\\
$=2^{l+2m+n}3^n-1$\\
$=2^{l+2m+n}3^n-2+1$\\
$=2(2^{l+2m+n-1}3^n-1)+1$\\
$=2\mu+1$\\
\\i.e. $N=2\mu+1$ where $m,n$ and $\mu$ are positive integers such that $\mu=2^{2n+m-1}-1$.\\
\\

From Lemmas 2.11 and Fl\o{}ystad \cite{4} corollary 1 the monad
\[
\begin{CD}
M_\bullet: 0@>>>\OO_{\PP^{2n+1}}(-1)^{\oplus{k}}@>>^{f}>{\OO^{\oplus{2n}\oplus{2k}}_{\PP^{2n+1}}} @>>^{g}>\OO_{\PP^{2n+1}}(1)^{\oplus{k}}  @>>>0\\
\end{CD}
\] 
exists and for a line bundle $\mathscr{L}=\OO_X(1,\cdots,1)$ we have the Segre embedding
\[
\xymatrix{
i^*:X = \PP^1\times\cdots\times\PP^1\times\PP^3\times\cdots\times\PP^3\times\PP^5\times\cdots\times\PP^5 \hookrightarrow\PP\big(H^0(X,\OO_X(1,\ldots,1))\big)\cong \PP^N}
\]
such that $i^*(\OO_X(1))\simeq \mathscr{L}$ and $N=2\mu+1$ supposing that one of the conditions of Lemma 2.13 is satified then the morphisms $A$ and $B$ in Lemma 2.14
induce the desired monad whose morpsims are $f$ and $g$.

In this case the monad 
\[\begin{CD}
0@>>>{\OO_{\PP^{2\mu+1}}(-1)^{\oplus{k}}} @>>^{A}>{\OO^{\oplus2\mu+2k}_{\PP^{2\mu+1}}}@>>^{B}>\OO_{\PP^{2\mu+1}}(1)^{\oplus{k}}@>>>0\end{CD}
\]
whose morphisms $A$ and $B$ that establish the monad are as given in the lemma above\\
\\
We induce a monad on $X:=\PP^1\times\cdots\times\PP^1\times\PP^3\times\cdots\times\PP^3\times\PP^5\times\cdots\times\PP^5$%
 \[\begin{CD}
M_\bullet: 0@>>>\OO_{X}(-1,\cdots,-1)^{\oplus{k}}@>>^{\overline{A}}>{\OO^{\oplus{2\mu}\oplus{2k}}_X} @>>^{\overline{B}}>\OO_{X}(1,\cdots,1)^{\oplus{k}}  @>>>0\\
\end{CD}\]
by giving the morphisms $\overline{A}$ and $\overline{B}$ with $\overline{B}\cdot\overline{A}=0$ and $\overline{A}$ and $\overline{B}$ are of maximal rank.\\
\\%
From $A$ and $B$ whose entries are $x_0,\cdots,x_{\mu},y_0,\cdots,y_{\mu}$ the homogeneous coordinates on $\PP^{2n+1}$ we give the correspondence
for the the Segre embedding using the following table:\\
\[ \begin{array}{|c|c|}
\hline
homog. coord. ~~on ~~\PP^{2n+1} & representation ~homog. coord.~~ on ~~X\\
\hline
x_0 & a_{00\cdots00}b_{00\cdots00}c_{00\cdots00} \\
x_1 & a_{00\cdots00}b_{00\cdots00}c_{00\cdots01} \\
x_2 & a_{00\cdots00}b_{00\cdots00}c_{00\cdots02} \\
x_3 & a_{00\cdots00}b_{00\cdots00}c_{00\cdots03} \\
x_4 & a_{00\cdots00}b_{00\cdots00}c_{00\cdots04} \\
\vdots&\vdots\\
x_{\mu-1} &  a_{01\cdots11}b_{33\cdots33}c_{55\cdots53} \\
x_\mu & a_{01\cdots11}b_{33\cdots33}c_{55\cdots54} \\
y_0 & a_{01\cdots11}b_{33\cdots33}c_{55\cdots55} \\
y_1 & a_{10\cdots00}b_{00\cdots00}c_{00\cdots00} \\
y_2 & a_{10\cdots00}b_{00\cdots00}c_{00\cdots01}  \\
y_3 & a_{10\cdots00}b_{00\cdots00}c_{00\cdots02}  \\
y_4 & a_{10\cdots00}b_{00\cdots00}c_{00\cdots003}  \\
\vdots&\vdots\\
y_{\mu-1} & a_{11\cdots11}b_{33\cdots33}c_{55\cdots54}\\
y_{\mu} & a_{11\cdots11}b_{33\cdots33}c_{55\cdots55}\\
\hline
\end{array} 
\]
\\
where $a_{ii\cdots ii}b_{jj\cdots jj}c_{kk\cdots kk}$ for $i\in\{0,1\}$, $j\in\{0,1,2,3\}$ and $k\in\{0,1,2,3,4,5\}$ are monomials of multidegree $(1,\ldots,1)$ i.e.\\

\[ \begin{array}{|c|c|}
\hline
homog. coord. ~~on ~~\PP^{2n+1} & homog. coord.~~ on ~~X\\
\hline
a_{00\cdots00}b_{00\cdots00}c_{00\cdots00} & \alpha_{10}\alpha_{20}\cdots\alpha_{(l-1)0}\alpha_{l0}\beta_{10}\beta_{20}\cdots\beta_{(m-1)0}\beta_{m0}\gamma_{10}\gamma_{20}\cdots\gamma_{(n-1)0}\gamma_{n0} \\
a_{00\cdots00}b_{00\cdots00}c_{00\cdots01} & \alpha_{10}\alpha_{20}\cdots\alpha_{(l-1)0}\alpha_{l0}\beta_{10}\beta_{20}\cdots\beta_{(m-1)0}\beta_{m0}\gamma_{10}\gamma_{20}\cdots\gamma_{(n-1)0}\gamma_{n1} \\
a_{00\cdots00}b_{00\cdots00}c_{00\cdots02} & \alpha_{10}\alpha_{20}\cdots\alpha_{(l-1)0}\alpha_{l0}\beta_{10}\beta_{20}\cdots\beta_{(m-1)0}\beta_{m0}\gamma_{10}\gamma_{20}\cdots\gamma_{(n-1)0}\gamma_{n2} \\
a_{00\cdots00}b_{00\cdots00}c_{00\cdots03} & \alpha_{10}\alpha_{20}\cdots\alpha_{(l-1)0}\alpha_{l0}\beta_{10}\beta_{20}\cdots\beta_{(m-1)0}\beta_{m0}\gamma_{10}\gamma_{20}\cdots\gamma_{(n-1)0}\gamma_{n3} \\
a_{00\cdots00}b_{00\cdots00}c_{00\cdots04} & \alpha_{10}\alpha_{20}\cdots\alpha_{(l-1)0}\alpha_{l0}\beta_{10}\beta_{20}\cdots\beta_{(m-1)0}\beta_{m0}\gamma_{10}\gamma_{20}\cdots\gamma_{(n-1)0}\gamma_{n4} \\
\vdots&\vdots\\
a_{01\cdots11}b_{33\cdots33}c_{55\cdots53} & \alpha_{10}\alpha_{21}\cdots\alpha_{(l-1)1}\alpha_{l1}\beta_{13}\beta_{23}\cdots\beta_{(m-1)3}\beta_{m3} \gamma_{15}\gamma_{25}\cdots\gamma_{(n-1)5}\gamma_{n3} \\
a_{01\cdots11}b_{33\cdots33}c_{55\cdots54} & \alpha_{10}\alpha_{21}\cdots\alpha_{(l-1)1}\alpha_{l1}\beta_{13}\beta_{23}\cdots\beta_{(m-1)3}\beta_{m3} \gamma_{15}\gamma_{25}\cdots\gamma_{(n-1)5}\gamma_{n4} \\
a_{01\cdots11}b_{33\cdots33}c_{55\cdots55} & \alpha_{10}\alpha_{21}\cdots\alpha_{(l-1)1}\alpha_{l1}\beta_{13}\beta_{23}\cdots\beta_{(m-1)3}\beta_{m3} \gamma_{15}\gamma_{25}\cdots\gamma_{(n-1)5}\gamma_{n5} \\
a_{10\cdots00}b_{00\cdots00}c_{00\cdots000} & \alpha_{11}\alpha_{20}\cdots\alpha_{(l-1)0}\alpha_{l0}\beta_{10}\beta_{20}\cdots\beta_{(m-1)0}\beta_{m0}\gamma_{10}\gamma_{20}\cdots\gamma_{(n-1)0}\gamma_{n0} \\
a_{10\cdots00}b_{00\cdots00}c_{00\cdots001} & \alpha_{11}\alpha_{20}\cdots\alpha_{(l-1)0}\alpha_{l0}\beta_{10}\beta_{20}\cdots\beta_{(m-1)0}\beta_{m0}\gamma_{10}\gamma_{20}\cdots\gamma_{(n-1)0}\gamma_{n1} \\
a_{10\cdots00}b_{00\cdots00}c_{00\cdots002} & \alpha_{11}\alpha_{20}\cdots\alpha_{(l-1)0}\alpha_{l0}\beta_{10}\beta_{20}\cdots\beta_{(m-1)0}\beta_{m0}\gamma_{10}\gamma_{20}\cdots\gamma_{(n-1)0}\gamma_{n2} \\
a_{10\cdots00}b_{00\cdots00}c_{00\cdots003} & \alpha_{11}\alpha_{20}\cdots\alpha_{(l-1)0}\alpha_{l0}\beta_{10}\beta_{20}\cdots\beta_{(m-1)0}\beta_{m0}\gamma_{10}\gamma_{20}\cdots\gamma_{(n-1)0}\gamma_{n3} \\
\vdots&\vdots\\
a_{11\cdots11}b_{33\cdots33}c_{55\cdots54} & \alpha_{11}\alpha_{21}\cdots\alpha_{(l-1)1}\alpha_{l1}\beta_{13}\beta_{23}\cdots\beta_{(m-1)3}\beta_{m3}\gamma_{15}\gamma_{25}\cdots\gamma_{(n-1)5}\gamma_{n4} \\
a_{11\cdots11}b_{33\cdots33}c_{55\cdots55} & \alpha_{11}\alpha_{21}\cdots\alpha_{(l-1)1}\alpha_{l1}\beta_{13}\beta_{23}\cdots\beta_{(m-1)3}\beta_{m3} \gamma_{15}\gamma_{25}\cdots\gamma_{(n-1)5}\gamma_{n5} \\
\hline
\end{array} 
\]
Specifically we define $\overline{A}$ and $\overline{B}$ as follows\\

\[ \overline{B} =\left[\begin{array}{cc}
B_1 & B_2 
         \end{array} \right]\] and

\[ \overline{A} =\left[\begin{array}{cc}
-A_1 \\ A_2      \end{array} \right]\]
where\\

\[ B_1 :=\left[ \begin{array}{cccc}
a_{00\cdots00}b_{00\cdots00}c_{00\cdots00}\cdots & a_{01\cdots11}b_{33\cdots33}c_{55\cdots54} \\
   \ddots&\ddots\\
    & a_{00\cdots00}b_{00\cdots00}c_{00\cdots00}\cdots a_{01\cdots11}b_{33\cdots33}c_{55\cdots54}
\end{array} \right]
\]
\\
\[ B_2 :=\left[ \begin{array}{cccc}
a_{01\cdots11}b_{33\cdots33}c_{55\cdots55}\cdots  & a_{11\cdots11}b_{33\cdots33}c_{55\cdots55} \\
   \ddots&\ddots\\
    & a_{01\cdots11}b_{33\cdots33}c_{55\cdots55}\cdots  a_{11\cdots11}b_{33\cdots33}c_{55\cdots55}
\end{array} \right]
\]

\begin{center} and \end{center}

 \[ A_1:=\left[ \begin{array}{cccc}
a_{01\cdots11}b_{33\cdots33}c_{55\cdots55}\cdots  & a_{11\cdots11}b_{33\cdots33}c_{55\cdots55} \\
   \ddots&\ddots\\
    & a_{01\cdots11}b_{33\cdots33}c_{55\cdots55}\cdots  a_{11\cdots11}b_{33\cdots33}c_{55\cdots55}
\end{array} \right]
\]

\[ A_2 :=\left[ \begin{array}{cccc}
a_{00\cdots00}b_{00\cdots00}c_{00\cdots00}\cdots & a_{01\cdots11}b_{33\cdots33}c_{55\cdots54} \\
   \ddots&\ddots\\
    & a_{00\cdots00}b_{00\cdots00}c_{00\cdots00}\cdots a_{01\cdots11}b_{33\cdots33}c_{55\cdots54}
\end{array} \right]
\]
\\
We note that
\begin{enumerate}
 \item $\overline{B}\cdot \overline{A} = 0$ and
 \\
 \item  The matrices $\overline{B}$ and $\overline{A}$ have maximal rank\\
\end{enumerate}
Hence we get the desired monad,
\[\begin{CD}
M_\bullet: 0@>>>\OO_{X}(-1,\cdots,-1)^{\oplus{k}}@>>^{\overline{A}}>{\OO^{\oplus{2\mu}\oplus{2k}}_X} @>>^{\overline{B}}>\OO_{X}(1,\cdots,1)^{\oplus{k}}  @>>>0\\
\end{CD}\]
\end{construct}

\begin{theorem}
 Let $X = \PP^1\times\cdots\times\PP^1\times\PP^3\times\cdots\times\PP^3\times\PP^5\times\cdots\times\PP^5$,  be a Cartesian product of $l$ copies 
of $\PP^1$, $m$ copies of $\PP^3$ and $n$ copies of $\PP^5$. There exists a monad of the form
 \[
\begin{CD}
M_\bullet: 0@>>>\OO_{X}(-1,\cdots,-1)^{\oplus{k}}@>>^{\overline{A}}>{\OO^{\oplus{2\mu}\oplus{2k}}_X} @>>^{\overline{B}}>\OO_{X}(1,\cdots,1)^{\oplus{k}}  @>>>0\\
\end{CD}
\]
where $l,m,n,k$ are positive integers and $\mu=2^{l+2m+n-1}3^n-1$.
\end{theorem}

\begin{proof} 
From the conditions of Lemma 2.1 (Fl\o{}ystad), we have $a=c=k$, $b=2\mu+2k$ and $\nu=\dim X$,  $X=\PP^1\times\cdots\times\PP^1\times\PP^3\times\cdots\times\PP^3\times\PP^5\times\cdots\times\PP^5$.\\
We show that these conditions hold. \\
Condition Two: $b\geq a+c+\nu$.\\
Now 
\begin{align*}
\begin{split}
b & = 2k+2\mu\\
  & = 2k+ 2[2^{l+2m+n-1}3^n-1]\\
  & = 2k+ 2^{l+2m+n}3^n-2\\
  & = 2k+ 2^l4^m6^n-2\\
  & > 2k+ 2^l+4^m+6^n\\
  & > 2k+ 2^l+3^m+5^n\\
  & > 2k+ l+3m+5n\\
  & = 2k + \dim X\\
  & = a+c+\nu
\end{split}
\end{align*}
Thus $b= 2k+2\mu\geq 2k + \dim X = a+c+\nu$.\\
Condition One follows since $b>a+c+\nu>a+c+\nu-1$ and $b=2k+2\mu\geq 2k = a+c$.\\
The morphisms $\overline{A}$ and $\overline{B}$ are constructed explicitly in the above construction 3.2.

\end{proof}

\begin{lemma}
Let $T$ be a vector bundle on $X = \PP^1\times\cdots\times\PP^1\times\PP^3\times\cdots\times\PP^3\times\PP^5\times\cdots\times\PP^5$ defined by the short exact sequence
\[\begin{CD}0@>>>T @>>>\OO_X^{2\mu+2k}@>>^{g}>\OO_X(1,\cdots,1)^{\oplus k} @>>>0\end{CD}\]
then $T$ is stable for an ample line bundle $\mathscr{L} = \OO_X(1,\cdots,1)$
\end{lemma}

\begin{proof}
We show that $H^0(X,(\bigwedge^q T)_{{\mathscr{L}}-norm}(f_1,\cdots,f_{l},g_1,\cdots,g_m,h_1,\cdots,h_m)) = 0$ for all
$\displaystyle{\sum_i^l f_i>0}$, $\displaystyle{\sum_j^m g_i>0}$, $\displaystyle{\sum_k^n h_k>0}$ and $1\leq q\leq \rk(T)$.\\
\\
Consider the ample line bundle $\mathscr{L} = \OO_X(1,\cdots,1) = \OO(L)$. Its class in $\pic(X)$ where\\
$\pic(X)= \langle [f_1\times\PP^1],\cdots,[\PP^1\times f_l], [g_1\times\PP^3],\cdots,[\PP^3\times g_m], [h_1\times\PP^5],\cdots,[\PP^5\times h_m]\rangle$
corresponds to the class $\sum_{i=1}^l1\cdot[f_i\times\PP^1]+\sum_{j=1}^m1\cdot[g_j\times\PP^3]+\sum_{k=1}^m1\cdot[h_k\times\PP^5]$\\
where $f_i$, $i=1,\cdots,l$ are hyperplanes of $\PP^1$ with the intersection product induced by $f_i^{1} = 1$ and $f_i^{2}=0$.\\
$g_j$, $j=1,\cdots,m$ are hyperplanes of $\PP^3$ with the intersection product induced by $g_j^{3} = 1$ and $g_j^{4}=0$.\\
$h_k$, $k=1,\cdots,n$ are hyperplanes of $\PP^5$ with the intersection product induced by $h_k^{5} = 1$ and $h_k^{6}=0$.\\
\\
From the short exact sequence \[\begin{CD}0@>>>T @>>>\OO_X^{2\mu+2k}@>>^{g}>\OO_X(1,\cdots,1)^{\oplus k} @>>>0\end{CD}\] we get
\[c_1(T) = (-k,\cdots,-k)\]
Since $L^{l+3m+5n}>0$, then degree of $T$ is given by
$\deg_{\mathscr{L}}T$\\$= -k(\sum_{i=1}^l[f_i\times\PP^1]+\sum_{j=1}^m[g_j\times\PP^3]+\sum_{k=1}^m[h_k\times\PP^5])\cdot\\ 
(\sum_{i=1}^l1\cdot[f_i\times\PP^1]+\sum_{j=1}^m1\cdot[g_j\times\PP^3]+\sum_{k=1}^m1\cdot[h_k\times\PP^5])^{l+3m+5n-1} = -\gamma L^{l+3m+5n}< 0$.\\
Since $\deg_{\mathscr{L}}T<0$, then $(\bigwedge^q T)_{\mathscr{L}-norm} = (\bigwedge^q T)$ and  it suffices by  Proposition 2.4, 
to prove that $h^0(\bigwedge^q T(-f_1,\cdots,-f_l,-g_1,\cdots,-g_m,-h_1,\cdots,-h_n)) = 0$ with $\displaystyle{\sum_{i=1}^lf_i>0}$, 
$\displaystyle{\sum_{j=1}^mg_j>0}$, and $\displaystyle{\sum_{k=1}^nh_k>0}$ for all $1\leq q\leq \rk(T)-1$.\\
\\
Next we twist the exact sequence 
\[\begin{CD}
0@>>>T @>>>{\OO_X^{2\mu+2k}} @>>>\OO_X(1,\cdots,1)^{\oplus k} @>>>0
\end{CD}\]
by $\OO_X (-f_1,\cdots,-f_l,-g_1,\cdots,-g_m,-h_1,\cdots,-h_n)$ we get the sequence,\\
\\
$0\lra T(-f_1,\cdots,-f_l,-g_1,\cdots,-g_m,-h_1,\cdots,-h_n)\lra\\\lra\OO_X(-f_1,\cdots,-f_l,-g_1,\cdots,-g_m,-h_1,\cdots,-h_n)^{\oplus{2\mu+2k}}\lra\\\lra\OO_X(1-f_1,\cdots,1-f_l,1-g_1,\cdots,1-g_m,1-h_1,\cdots,1-h_n)^{\oplus k}\lra0$\\
\\
and taking the exterior powers of the sequence by Proposition 2.5 we obtain\\
\\
$0\lra \bigwedge^q T(-f_1,\cdots,-f_l,-g_1,\cdots,-g_m,-h_1,\cdots,-h_n) \lra\\\lra \bigwedge^q (\OO_X(-f_1,\cdots,-f_l,-g_1,\cdots,-g_m,-h_1,\cdots,-h_n)^{\oplus{2\mu+2k}})\lra\\\lra \bigwedge^{q-1}(\OO_X(1-2f_1,\cdots,1-2f_l,1-2g_1,\cdots,1-2g_m,1-2h_1,\cdots,1-2h_n)^{\oplus k})\cdots$\\
\\
Taking cohomology we have the injection:\\
$0\lra H^0(\bigwedge^{q}T(-f_1,\cdots,-f_l,-g_1,\cdots,-g_m,-h_1,\cdots,-h_n))\lra\\\hookrightarrow H^0(\bigwedge^{q}(\OO_X(-f_1,\cdots,-f_l,-g_1,\cdots,-g_m,-h_1,\cdots,-h_n))^{\oplus k}$\\
\\%
since $\displaystyle{\sum_i^l f_i>0}$, $\displaystyle{\sum_i^m g_j>0}$ and  $\displaystyle{\sum_i^n h_k>0}$ using Lemma 2.9, Lemma 2.10 and Theorem 2.8
then \[h^0(X,\bigwedge^{q}(\OO_X(-f_1,\cdots,-f_l,-g_1,\cdots,-g_m,-h_1,\cdots,-h_n)^{\oplus k}))= 0\]
thus it follows $h^0(\bigwedge^{q}T(-f_1,\cdots,-f_l,-g_1,\cdots,-g_m,-h_1,\cdots,-h_n))=0$ and hence $T$ is stable.

\end{proof}

\begin{theorem} Let $X = \PP^1\times\cdots\times\PP^1\times\PP^3\times\cdots\times\PP^3\times\PP^5\times\cdots\times\PP^5$, then the cohomology vector bundle $E$ associated to the monad 
\[
\begin{CD}
0@>>>{\OO_X(-1,\cdots,-1)^{\oplus k}} @>>^{f}>{\OO_X^{2\mu+2k}}@>>^{g}>\OO_X(1,\cdots,1)^{\oplus k} @>>>0
\end{CD}
\]
of rank $2\mu$ is simple.
\end{theorem}

\begin{proof}
The display of the monad is
\[
\begin{CD}
@.@.0@.0\\
@.@.@VVV@VVV\\
0@>>>{\OO_{X}(-1,\cdots,-1)^{\oplus k}} @>>>T=\ker(\beta)@>>>E=\ker(\beta)/\im(\alpha)@>>>0\\
@.||@.@VVV@VVV\\
0@>>>{\OO_{X}(-1,\cdots,-1)^{\oplus k}} @>>^{\alpha}>{\OO^{\oplus2\mu+2k}_{X}}@>>>Q=\cok(\alpha)@>>>0\\
@.@.@V^{\beta}VV@VVV\\
@.@.{\OO_{X}(1,\cdots,1)^{\oplus k} }@={\OO_{X}(1,\cdots,1)^{\oplus k} }\\
@.@.@VVV@VVV\\
@.@.0@.0
\end{CD}
\]

\noindent To show that the cohomology vector bundle $E$ with rank $2\mu$ is simple, we shall rely on the stability of the bundle $T$ 
and analysis of the short exact sequences making up the display of the monad.\\
\\
Take the dual of the short exact sequence that appears as the first row of the display diagram of the monad and tensor it by  $E$ to obtain
\[
\begin{CD}
0@>>>E\otimes E^* @>>>E\otimes T^* @>>>E(1,\cdots,1)^k@>>>0\end{CD}
\]
and on taking cohomology it follows
\begin{equation}
h^0(X,E\otimes E^*) \leq h^0(X,E\otimes T^*)
\end{equation}
\\
We now dualize the exact sequence the forms the first column of the display diagram of the monad to obtain
\[\begin{CD}
0@>>>\OO_X(-1,\cdots,-1)^{\oplus k} @>>>{\OO_X^{\oplus2\mu+2k}} @>>>T^* @>>>0
\end{CD}\]
which on twisting by $\OO_X(-1,\cdots,-1)$, taking cohomology and applying Lemma 2.7 and Theorem 2.8 we deduce 
$H^0(X,T^*(-1,\cdots,-1)) = H^1(X,T^*(-1,\cdots,-1)) = 0$.\\
\\
Finally on tensoring the sequence on the first row of the display, taking cohomology and since $H^1(X,T^*(-1,\cdots,-1)^k=0$ for $k>1$ from the above lemma 
and from (1) above so we have 
\[1\leq h^0(X,T\otimes T^*)\leq h^0(X,E\otimes E^*) \leq h^0(X,E\otimes T^*)\leq1\]
It thus follows $ h^0(X,E\otimes E^*) = 1 $ and thus $E$ is simple.

\end{proof}

\section{Monads on Cartesian products of even and odd projective spaces}

\begin{theorem}
Let $X = \PP^{a_1}\times\cdots\times\PP^{a_n}$ where $a_1\leq a_2\leq\cdots\leq a_n$ alternating even and odd or at least $a_i$ is odd for $0<i\leq n$ and $\mathscr{L}=\OO_X(1,\cdots,1)$ an ample line bundle.
Denote by $N = h^0(\OO_X(1,\cdots,1)) - 1$. Then there exists a linear monad $M_\bullet$ on $X$ of the form
\[
\begin{CD}
M_\bullet: 0@>>>\OO_{X}(-1,\cdots,-1)^{\oplus\alpha}@>>^{\overline{A}}>{\OO^{\oplus\beta}_X} @>>^{\overline{B}}>\OO_{X}(1,\cdots,1)^{\oplus\gamma}  @>>>0
\end{CD}
\]
if atleast one of the following is satified
\begin{enumerate}
\renewcommand{\theenumi}{\alph{enumi}}
 \item $\beta\geq 2\gamma + N -1$, and $\beta\geq \alpha + \gamma$,
 \item $\beta\geq \alpha + \gamma + N$, where $\alpha,\beta, \gamma$ be positive integers. 
\end{enumerate}
\end{theorem}

\begin{proof}
For the ample line bundle $\mathscr{L} = \OO_X(1,\ldots,1)$ we have the Segre embedding
\[
\xymatrix{
i^*:X = \PP^{a_1}\times\cdots\times\PP^{a_n} \hookrightarrow\PP^{N}
}
\]
such that $i^*(\OO_X(1))\simeq \mathscr{L}$ 
\\
Since $a_1<\cdots<a_n$ alternating even and odd we can suppose $a_i$ is odd then $a_i+1$ is even i.e. $a_i+1=2\theta$, $\theta$ a positive integer.\\
Then we have $\displaystyle{N=\prod_{i=1}^n(a_i+1)-1}$
\begin{align*}
\begin{split}
N & = (a_1+1)\cdots(a_i+1)\cdots(a_n+1)-1\\
  & = (a_1+1)\cdots2\theta\cdots(a_n+1)-1\\
  & = 2(a_1+1)\cdots\theta\cdots(a_n+1)-1\\
  & = 2[(a_1+1)\cdots\theta\cdots(a_n+1)]-1\\
  \end{split}
\end{align*}
Thus $N$ is odd that is $N=2n+1$ for $n$ a positive integer.\\
Now suppose that one of the conditions of Lemma 2.11 is satified and we have 
$a=\alpha$, $b=\beta$, $c=\gamma$ and $\nu=2n+1$ thus there exists a linear monad
\[
\begin{CD}
0@>>>{\OO_{\PP^{2n+1}}(-1)^{\oplus\alpha}} @>>^{A}>{\OO^{\oplus\beta}_{\PP^{2n+1}}} @>>^{B}>{\OO_{\PP^{2n+1}}(1)^{\oplus\gamma}} @>>>0\\
\end{CD}
\]
on $\PP^{2n+1}$ whose morphisms are matrices $A$ and $B$ with entries monomials of degree one where
\begin{align*}
A\in\Hom(\OO_{\PP^{2n+1}}(-1)^{\oplus\alpha},\OO_{\PP^{2n+1}}^{\oplus\beta})\cong H^0(\PP^{2n+1},\OO_{\PP^{2n+1}}(1)^{\oplus\alpha\beta}) \\
B\in\Hom(\OO_{\PP^{2n+1}}^{\oplus\beta},\OO_{\PP^{2n+1}}(1)^{\oplus\gamma})\cong H^0(\PP^{2n+1},\OO_{\PP^{2n+1}}(1)^{\oplus\beta\gamma})
\end{align*}
Thus, $A$ and $B$ of the form
\[ B :=\left( \begin{array}{cccc|cccccccc}
x_0\cdots  & x_{n} &       &   &y_0 \cdots  & y_{n}\\
    &\ddots&\ddots &&\ddots&\ddots\\
    && x_0\cdots   x_{n} & & & y_0 \cdots  & y_{n}
\end{array} \right)
\]
 and
\[ A :=\left( \begin{array}{cccccccc}
-y_0\cdots  & -y_{n} \\
    	   &\ddots &\ddots\\
             &&-y_0 \cdots & -y_{n}\\
\hline
x_0 \cdots  & x_{n} \\
    	    &\ddots &\ddots\\
             && x_0\cdots & x_{n}\\
\end{array} \right)
\]

induce the expected monad on $X$, 
\[
\begin{CD}
M_\bullet: 0@>>>\OO_{X}(-1,\cdots,-1)^{\oplus\alpha}@>>^{\overline{A}}>{\OO^{\oplus\beta}_X} @>>^{\overline{B}}>\OO_{X}(1,\cdots,1)^{\oplus\gamma}  @>>>0
\end{CD}
\]

\end{proof}

\vspace{0.5cm}
\begin{construct}
Let $\psi : X = \PP^1\times\PP^2\times\PP^3\longrightarrow \PP^{23}$ be the segre embedding  defined by\\
\\
$[a_0:a_1][b_0:b_1:b_2][c_0:c_1:c_2:c_3]\hookrightarrow [x_0:\ldots:x_{11}:y_0:\dots:y_{11}]$.\\
\\
Then, by Fl\o{}ystad\cite{4}, there exists a linear monad
\[
\begin{CD}
0@>>>\OO_{\PP^{23}}(-1)^{\alpha} @>>^{A}>{\OO^{\beta}_{\PP^{23}}} @>>^{B}>\OO_{\PP^{23}}(1)^{\gamma}  @>>>0\\
\end{CD}
\]

\[ B :=\left( \begin{array}{cccc|cccccccc}
x_0\cdots  & x_{11} &       &   &y_0 \cdots  & y_{11}\\
    &\ddots&\ddots &&\ddots&\ddots\\
    && x_0\cdots   x_{11} & & & y_0 \cdots  & y_{11}
\end{array} \right)
\]
 and
\[ A :=\left( \begin{array}{cccccccc}
-y_0\cdots  & -y_{11} \\
    	   &\ddots &\ddots\\
             &&-y_0 \cdots & -y_{11}\\
\hline
x_0 \cdots  & x_{11} \\
    	    &\ddots &\ddots\\
             && x_0\cdots & x_{11}\\
\end{array} \right)
\]

Now we induce the monad
\[
\begin{CD}
0@>>>\OO_{X}(-1,-1,-1)^{\alpha} @>>^{\overline{A}}>{\OO^{\beta}_{X}} @>>^{\overline{B}}>\OO_{X}(1,1,1)^{\gamma}  @>>>0\\
\end{CD}
\]
We construct $\overline{A}$ and  $\overline{B}$ from $A$ and $B$ from the the segre map using the table:\\
\[ \begin{array}{|c|c|}
\hline
A,B $entries$ & \overline{A}, \overline{B}, entries\\
\hline
x_0 & a_0b_0c_0 \\
x_1 & a_0b_0c_1 \\
x_2 & a_0b_0c_2 \\
x_3 & a_0b_0c_3 \\
x_4 & a_0b_1c_0 \\
x_5 & a_0b_1c_1 \\
x_6 &a_0b_1c_2 \\
x_7 & a_0b_1c_3 \\
x_8 & a_0b_2c_0 \\
x_9 & a_0b_2c_1 \\
x_{10} & a_0b_2c_2 \\
x_{11} & a_0b_2c_3 \\
y_0 & a_1b_0c_0 \\
y_1 & a_1b_0c_1 \\
y_2 & a_1b_0c_2 \\
y_3 & a_1b_0c_3 \\
y_4 & a_1b_1c_0 \\
y_5 & a_1b_1c_1 \\
y_6 & a_1b_1c_2 \\
y_7 & a_1b_1c_3 \\
y_8 & a_1b_2c_0 \\
y_9 & a_1b_2c_1 \\
y_{10} & a_1b_2c_2 \\
y_{11} & a_1b_2c_3 \\
\hline
\end{array} 
\]
\\
Specifically we define two matrices $\overline{A}$ and $\overline{B}$ as follows\\
\[ \overline{B} =\left( \begin{array}{c|c}
B_1 & B_2
         \end{array} \right)\] and

\[ \overline{A}=\left( \begin{array}{cc}
A_1 \\ A_2
         \end{array} \right)\]
Where 

\[ B :=\left( \begin{array}{cccc|cccccccc}
a_0b_0c_0\cdots  & a_0b_2c_3 &       &   &a_1b_0c_0 \cdots  & a_1b_2c_3\\
    &\ddots&\ddots &&\ddots&\ddots\\
    && a_0b_0c_0\cdots   a_0b_2c_3 & & & a_1b_0c_0 \cdots  & a_1b_2c_3
\end{array} \right)
\]
 and
\[ A :=\left( \begin{array}{cccccccc}
-a_1b_0c_0\cdots  & -a_1b_2c_3 \\
    	   &\ddots &\ddots\\
             &&-a_1b_0c_0 \cdots & -a_1b_2c_3\\
\hline
a_0b_0c_0\cdots  & a_0b_2c_3 \\
    	    &\ddots &\ddots\\
             &&a_0b_0c_0\cdots & a_0b_2c_3\\
\end{array} \right)
\]

We note that
\begin{enumerate}
 \item $\overline{B}\cdot \overline{A} = 0$ and
 \\
 \item  The matrices $\overline{B}$ and $\overline{A}$ have maximal rank\\
\end{enumerate}
Hence we get the desired monad,\\
\[
\begin{CD}
0@>>>\OO_{X}(-1,-1,-1)^{\alpha} @>>^{\overline{A}}>{\OO^{\beta}_{X}} @>>^{\overline{B}}>\OO_{X}(1,1,1)^{\gamma}  @>>>0\\
\end{CD}
\]
\end{construct}

\section{Acknowledgment}
\noindent I must acknowledge where journey of constructing vector bundles started, way back in 2012 in Barcelona thanks to the unfailing 
support of Laura Costa of the Universitat de Barcelona, thanks a million! I am extremely grateful to Dr. Melissa Muindi, my dear wife and to Kavete, Maingi and Wachuka our wonderful kids who are always supportive of my pursuits.

\vspace{1cm}

\noindent \textbf{Data Availability statement}
My manuscript has no associate data.

\vspace{1cm}

\noindent \textbf{Conflict of interest}
On behalf of all authors, the corresponding author states that there is no conflict of interest.

\end{document}